\documentclass{article}
\usepackage{amsmath,amssymb,amsthm,color}

\theoremstyle{definition}
\newtheorem{theorem}{Theorem}[section]
\newtheorem{lemma}[theorem]{Lemma}
\newtheorem{remark}[theorem]{Remark}
\newtheorem{definition}[theorem]{Definition}

\newtheorem{proposition}[theorem]{Proposition}

\makeatletter

\@addtoreset{equation}{section}
\makeatother

\title{A note on the volume entropy of harmonic manifolds of hypergeometric type\thanks{This paper has been submitted to Results in Mathematics.}}

\author{Hiroyasu Satoh
\thanks{Liberal Arts and Sciences, Nippon Institute of Technology, 4-1 Gakuendai, Miyashiro-machi, Minamisaitama-gun, Saitama 345-8501 JAPAN\quad
e-mail : hiroyasu@nit.ac.jp}}

\date{\today}

\begin{document}

\maketitle

\begin{abstract}
Harmonic manifolds of hypergeometric type form a class of non-compact harmonic manifolds that includes rank one symmetric spaces of non-compact type and Damek-Ricci spaces.
When normalizing the metric of a harmonic manifold of hypergeometric type to satisfy the Ricci curvature $\mathrm{Ric} = -(n-1)$, we show that the volume entropy of this manifold satisfies a certain inequality.
Additionally, we show that manifolds yielding the upper bound of volume entropy are only real hyperbolic spaces with sectional curvature $-1$, while examples of Damek-Ricci spaces yielding the lower bound exist in only four cases.
\smallskip

\noindent\textbf{Keywords:} harmonic manifold, volume entropy, Damek-Ricci space
\smallskip

\noindent\textbf{MSC Classification:} 53C21, 53C25, 53C20 
\end{abstract}

\section{Introduction and Main Results}

Let $(M, g)$ be a complete and simply connected Riemannian manifold.
A manifold $(M, g)$ is called a harmonic manifold when expressed in normal coordinates centered at an arbitrary point $p \in M$, the volume density function $\sqrt{\det(g_{ij})}$ depends only on the distance $r = d(p, \cdot)$ from the center $p$.
The study of harmonic manifolds originated from the existence problem of solutions to the Laplace equation $\Delta f = 0$ that are non-constant and depend only on the distance from an arbitrary fixed point.
Here $\Delta$ is the Laplace-Beltrami operator on $M$.
There are several equivalent statements for the definition of harmonic manifolds.
Please refer to \cite{RWW,Besse,Sz} for details.
The Euclidean spaces and the rank one symmetric spaces are typical harmonic manifolds, suggesting the local symmetricity conjecture posed by Lichnerowicz \cite{L}.
Walker \cite{Wa} proved that the conjecture holds true in the case of four dimensions.
Moreover, it is known that 5-dimensional harmonic manifolds are spaces of constant curvature (\cite[Theorem 1]{Ni}).
The conjecture is affirmative in the compact case from the results by \cite{Sz, BCG}.
However, there exist non-symmetric harmonic manifolds among the class of Damek-Ricci spaces \cite{DR}.
A Damek-Ricci space is a one-dimensional extension of a generalized Heisenberg group equipped with a left-invariant Riemannian metric.
These spaces constitute a class of harmonic Hadamard manifolds that includes all rank-one symmetric spaces of non-compact type except real hyperbolic spaces.
The lowest dimension for a non-symmetric Damek-Ricci space is seven  \cite[p.110]{BTV}.
Furthermore, it is known that homogeneous harmonic manifolds are either flat Euclidean spaces,
rank-one symmetric spaces, or Damek-Ricci spaces \cite{He}.
Harmonic manifolds of polynomial volume growth are flat \cite{RS}.
In \cite{K}, it is claimed that harmonic manifolds with subexponential volume growth also become flat spaces.
From the above, the existence of non-compact, non-homogeneous harmonic manifolds with exponential volume growth in dimensions $\ge 6$ remains an open problem (see \cite{K}, \cite[p.110]{BTV}).

In \cite{Itoh-S2019},
M. Itoh and the author defined spaces of \textit{hypergeometric type}.
These are spaces where the eigenfunctions of the Laplace-Beltrami operator can be described by hypergeometric functions,
and they can also be interpreted as spaces where the theory of the spherical Fourier transform works well.
This means that inverse transforms can be defined, and properties such as the Plancherel theorem hold true.
Let $f(x)$ be a smooth radial function of compact support on $M$, i.e., there exists a point $p\in M$ and an even function $F : \mathbb{R}\rightarrow \mathbb{R}$ such that $f(x)=F(r(x))$, where $r=d(p,\,\cdot\,)$ is the distance function from $p$. 
Then, we define the spherical Fourier transform $\hat{f} : \mathbb{C} \rightarrow \mathbb{C}$ of $f$ by
\begin{equation}\label{sphrclFT}
\hat{f}(\lambda)=\int_{x\in M} f(x)\,\varphi_\lambda(x)\,dv_M(x)
=\omega_{n-1} \int_0^\infty F(r)\,\Phi_\lambda(r)\,\Theta(r)\,dr,
\end{equation}
where $\Theta(r)$ is the volume density function of the geodesic sphere with radius $r$ and $\omega_{n-1}= \mathrm{vol}(S^{n-1}(1))=\dfrac{2\sqrt{\pi^n}}{\Gamma\left(\frac{n}{2}\right)}$.
The function $\varphi_\lambda(x)=\Phi_\lambda(r(x))$ is called a \textit{spherical function} which is a radial eigenfunction of the Laplace-Beltrami operator $\Delta$ on $M$, normalized by $\varphi_\lambda(p)=\Phi_\lambda(0)=1$.
Anker et al. \cite{ADY} state that the spherical Fourier transforms on Damek-Ricci spaces fit into the general framework of the Jacobi transform (see \cite{Koo}).
In this framework, various analyses become possible due to the representation of spherical functions by hypergeometric functions.
This arises from the fact that the eigenfunction equation of the radial part of the Laplace-Beltrami operator on a Damek-Ricci space
\begin{equation}\label{LEeq}
(\,\Delta^{\mathrm{rad}} \Phi(r) =\,)\ -\left(\dfrac{d^2}{dr^2}+\sigma(r)\dfrac{d}{dr}\right) \Phi(r)=\left(\dfrac{Q^2}{4}+\lambda^2\right) \Phi(r)
\end{equation}
turns into the hypergeometric equation
\begin{equation}\label{HGeq}
z(1-z) \dfrac{d^2f}{dz^2}(z)+\left\{c-(a+b+1)z\right\}\dfrac{df}{dz}(z)-a b f(z)=0
\end{equation}
when the variables are changed.
Motivated by this fact, we gave the definition of hypergeometric type (Definition \ref{HMHGT}).
Harmonic manifolds of hypergeometric type form a class of non-compact harmonic manifolds that includes rank one symmetric spaces of non-compact type and Damek-Ricci spaces.
The definition of hypergeometric type is equivalent to the volume density function being describable in a certain form (Theorem  \ref{volumedensmeancurv}, \ref{voldensfctheoremHGT}).

In this note, we improve the definition slightly to ensure that the properties of harmonic manifolds of hypergeometric type are not compromised even when the metric is rescaled.
We show that the volume entropy is bounded by certain constants from above and below.
The volume entropy is not invariant under rescaling of the metric.
Therefore, we normalize the metric with respect to the Ricci tensor and then evaluate the volume entropy.
Our main theorem is as follows.

\begin{theorem}\label{maintheorem}
Let $(M^n, g)$ be a non-compact harmonic manifold of dimension $n$, whose volume entropy $Q$ is positive.
Assume that $(M, g)$ is of hypergeometric type and the metric is rescaled so that $\mathrm{Ric}=-(n-1)$.
Then we show that $Q$ satisfies
\begin{equation}\label{volumeentropy-1}
\frac{2\sqrt{2}}{3}(n-1)\le Q \le (n-1)
\end{equation}
and $Q = (n-1)$ holds if and only if $(M, g)$ is isometric to the real hyperbolic space $\mathbb{R}H^n(-1)$ of constant sectional curvature $-1$.
\end{theorem} 

\begin{remark}
The general characterization of harmonic manifolds of hypergeometric type for which the volume entropy satisfies $Q=\frac{2\sqrt{2}}{3}(n-1)$ remains unknown.
However, considering the context of Damek-Ricci spaces, 
it is verified that there are only four examples that satisfy that equation.
This will be discussed in section \ref{DR_4ex}.
\end{remark}

From the above considerations, two problems arise.
One is whether there exist non-compact harmonic manifolds with $Q \neq 0$ that are not of hypergeometric type.
Given that the value of $\frac{2\sqrt{2}}{3}$ is very close to 1,
there could exist harmonic manifolds satisfying $\mathrm{Ric}=-(n-1)$ and $0 < Q < \frac{2\sqrt{2}}{3}(n-1)$.
If the existence of harmonic manifolds with volume entropy $Q$ in this range could be shown,
it would represent new examples of harmonic manifolds not previously known.
On the other hand, it is also worth considering the possibility of an isolation phenomenon regarding volume entropy.
Specifically, one may ask whether there exists a constant $\bar{Q}_n$ satisfying the following property:
There exists no harmonic manifold $(M^n, g)$ satisfying $\mathrm{Ric}=-(n-1)$ with volume entropy $Q$ such that $0 < Q < \bar{Q}_n$.
At present, this remains an open question.
In this paper, we don't delve into a detailed analysis of these issues but merely mention them as possible directions for future investigation.

\begin{remark}
It should be noted that the discussion in \cite{Itoh-S2019, Itoh-S2020} assumes that the manifold is Hadamard, but it is worth mentioning that the arguments can be extended to more general non-compact harmonic manifolds.
\end{remark}

This note is organized as follows.
In Section 2, basic preliminaries for Riemannian manifolds, the Busemann function and harmonic manifolds are given.
In Section 3, we present a revised definition of harmonic manifolds of hypergeometric type and show that the density function can be expressed in a certain form.
In Section 4, we provide the proof of the main theorem,
and in Section 5, we discuss the existence of Damek-Ricci spaces that provide the lower bound for \eqref{volumeentropy-1}.

\section{Preliminaries}

\subsection{Riemannian manifolds and geodesic spheres}\label{S21_Riemmfd}

We recall basic geometric notions of a Riemannian manifold.
Let $(M^n,g)$ be an $n$-dimensional complete Riemannian manifold and $\nabla$ be the Levi-Civita connection of $g$.
The Riemannian curvature tensor and the Ricci tensor of $g$ are denoted by
\begin{equation*}
R(X, Y)Z=\nabla_X(\nabla_Y Z)-\nabla_Y(\nabla_X Z)-\nabla_{[X, Y]}Z,
\end{equation*}
\begin{equation*}
\mathrm{Ric}(Y, Z)=\mathrm{trace}\{X\mapsto R(X, Y)Z\}.
\end{equation*}

\begin{remark}
The Ricci tensor is invariant under the rescaling of the metric.
That is, if we denote the Ricci tensor of $g$ as $\mathrm{Ric}_g$, then $\mathrm{Ric}_g=\mathrm{Ric}_{c^2 g}$ holds for any $c>0$.
\end{remark}

Let $\{r, \theta^i,\, i= 1,\ldots,n-1\}$ be geodesic polar coordinates around $p\in M$.
The Laplace-Beltrami operator $\Delta$ at a point $q= \exp_p r u$, $u\in U_p M$ is represented by
\begin{equation}\label{radiallaplace}
\Delta = - \left( \frac{\partial^2}{\partial r^2} + \sigma_p(\exp_p ru) \frac{\partial}{\partial r}\right) + {\tilde{\Delta}}
\end{equation}
where $\sigma_p(q)$ is the mean curvature of the geodesic sphere $S(p;r)$ at $q$ with respect to the inward unit normal vector field $\xi=-\nabla r$ and ${\tilde{\Delta}}$ is the Laplace-Beltrami operator on $S(p;r)$ at $q$ (see \cite[(1.2)]{Sz}).
Here $U_p M$ denotes the set of all unit tangent vectors at $p$.

Let $\gamma$ be a geodesic given by $\gamma(t):= \exp_p tu$, where $u\in U_p M$  and let $\{E_i= E_i(t), \, i=1,\ldots,n\}$ be a parallel orthonormal frame field along $\gamma$ with $E_1(t) = \gamma'(t)$.
Here the prime means the covariant derivative along $\gamma$.
Let $Y_i(t)$, $i=2,\ldots,n$ be a perpendicular Jacobi vector field along $\gamma$ for $t\geq 0$ satisfying $Y_i(0)=0$ and $Y'_i(0)=E_i(0)$.
Here a Jacobi field $Y(t)$ along $\gamma$ is the vector field along $\gamma$ which satisfies $Y''(t)+R(t) Y(t)=O$, where $R(t):=R_{\gamma'(t)}$ is the Jacobi operator $X\mapsto R(X, \gamma'(t))\gamma'(t)$.
Then, the square root determinant
\begin{equation}\label{squarerootdet}
\Theta_p(\exp_p tu) := 
\sqrt{\det\left(g(Y_i(t),Y_j(t)) \right)}_{2\leq i,j\leq n}
\end{equation}
yields the volume density of $S(p;t)$ (see \cite[p.166--167]{GHL}).

\begin{lemma}\label{dens_meancurv}
We assume that $p=\gamma(0)$ and $\gamma(t)=\exp_p(tu)$ for any $t\in (0, T)$ are not conjugate along $\gamma$.
Then, we have
\begin{equation*}
\sigma_p(\exp_p tu) = \frac{\frac{\partial}{\partial t}\Theta_p(\exp_p tu)}{\Theta_p(\exp_p tu)},\quad
0< t < T.
\end{equation*}
\end{lemma} 

\begin{proof}
Let $\gamma(t)^{\perp}=\{ v\in T_{\gamma(t)}M\, ;\, v \perp \gamma'(t)\}$ and $A(t)$ be an endomorphism of $\gamma^{\perp}(t)$, defined by $A(t) E_i(t) = Y_i(t)$, $i=2,\ldots,n$.
Then, $A(t)$ satisfies 
\begin{equation}\label{jacobitensor}
A''(t) + R(t)\circ A(t) = 0,\quad A(0) = 0,\quad A'(0) = {\rm Id}_{u^\perp}.
\end{equation} 
Using $A(t)$ we define an endomorphism $S(t) := A'(t)\circ A^{-1}(t)$ of $\gamma(t)^{\perp}$.
Then, $S(t)$ is the shape operator of $S(p;t)$ at $\gamma(t)$ (see \cite[p.6--8]{EsOs} for details).
Hence its trace ${\rm Tr} S(t) =: \sigma_p(\exp_p tu)$ is the mean curvature of $S(p;t)$.
Now from \eqref{squarerootdet} one obtains $\Theta_p(\exp_p tu) = \det A(t)$, so that
\begin{equation*}\label{det}
\dfrac{\partial}{\partial t}\det A(t) = \det A(t)\,{\rm Tr} (A'(t)\circ A^{-1}(t))
\end{equation*}
showing the lemma.
\end{proof}

\begin{lemma}\label{derivativesofq}
The volume density $\Theta_p$ of a geodesic sphere $S(p; t)$ satisfies
\begin{equation}\label{ledgerformula0}
\lim_{t\rightarrow 0} \frac{\Theta_p(\exp_p tu)}{t^{n-1}} = 1,\quad
\lim_{t\rightarrow 0} \frac{\partial}{\partial t}\, \left(\frac{\Theta_p(\exp_p tu)}{t^{n-1}}\right) = 0,
\end{equation}
and
\begin{equation}\label{ledgerformulaelse}
\left.\frac{\partial^2}{\partial t^2}\left( \frac{\Theta_p(\exp_y tu)}{t^{n-1}}\right)\right\vert_{t=0} = - \frac{1}{3} {\rm{Ric}}(u, u)
\end{equation}
for any $u\in U_p M$.
\end{lemma} 

For the formula \eqref{ledgerformulaelse} refer to \cite[p.1]{Sz}.

\begin{proof}
From  \eqref{jacobitensor}, the MacLaurin expansion of $A(t)$ yields
\begin{equation}\label{expansionofa}
A(t) = t\left({\rm Id}_{u^{\perp}} - \frac{1}{3!} R_u t^2 + O(t^3)\right),
\end{equation}
and thus $\Theta_p(\exp_ tu)$ can be expanded as
\begin{equation}\label{expand}
\Theta_p(\exp_p tu)=\det A(t) = t^{n-1}\left(1- \frac{1}{3!} \mathrm{Ric}(u,u) t^2 + O(t^2) \right),
\end{equation}
from which one gets the formulas of the lemma. 
\end{proof}
 

\begin{lemma}
The mean curvature $\sigma_p$ of a geodesic sphere $S(p; t)$ satisfies
\begin{equation}\label{meancurvature}
\lim_{t\rightarrow 0}\left(\sigma_p(\exp_p tu) - \frac{n-1}{t}\right) = 0.
\end{equation}
\end{lemma}


\begin{proof}
By setting $\tau_p(\exp_p tu) = \sigma_p(\exp_p tu) - \dfrac{n-1}{t}$ and $F(t)=\dfrac{\Theta_p(\exp_p tu)}{t^{n-1}}$,
one has
\begin{equation*}
\tau_p(\exp_p tu) =\dfrac{\partial}{\partial t}\log F(t)=\dfrac{F'(t)}{F(t)}.
\end{equation*}
From \eqref{ledgerformula0}, we obtain the lemma.
\end{proof}

\begin{definition}
The volume entropy of $(M, g)$ is defined by 
\begin{equation}\label{def_volent}
Q= \lim_{r\rightarrow\infty} \frac{1}{r} \log {\rm Vol}\, B(p;r),
\end{equation}
where ${\rm Vol}\, B(p;r)$ is the volume of the closed ball $B(p;r)=\{q\in M\,|\,d(p, q)\le r\}$;
\begin{equation}\label{volBall}
\mathrm{Vol} B(p;r) = \int_0^r dt \int_{u\in U_p M} \Theta_p(\exp_p tu)\,du.
\end{equation} 
\end{definition}

\begin{lemma}\label{Q_rescaling}
If we denote the volume entropy of $(M, g)$ as $Q_g$, then under a rescaling of the metric, we have $Q_{c^2 g}=Q_g/c$.
\end{lemma}

\begin{proof}
Under a rescaling of the metric, the volume form satisfies $dv_{c^2g}=c^n dv_{g}$,
and a ball $B_{c^2 g}(p; r)$ with radius $r$ with respect to $c^2 g$ becomes a ball $B_{g}(p; \frac{r}{c})$ with radius $\frac{r}{c}$ with respect to $g$.
Therefore, we have
\begin{align*}
Q_{c^2 g}
=&\lim_{r\rightarrow \infty}\frac{\log \int_{B_{c^2 g}(p; r)}\,dv_{c^2 g}}{r}\\
=&\lim_{r\rightarrow \infty}\frac{\log \int_{B_{g}(p; \frac{r}{c})}c^n\,dv_{g}}{r}\\
=&\lim_{r\rightarrow \infty}\frac{\log \int_{B_{g}(p; \frac{r}{c})}\,dv_{g}}{\frac{r}{c}}\cdot\dfrac{1}{c}
=\dfrac{1}{c}\,Q_g.
\end{align*}
\end{proof}

\subsection{Busemann function and horospheres}\label{S_busemann}

In this subsection,
we introduce the properties of the Busemann function and horospheres, which are necessary for our argument.
For further details, please refer to \cite{Inn, Esch, EsOs}.

Let $(M^n, g)$ be a complete and simply connected Riemannian manifold.
We denote the geodesic parametrized by its arc-length with initial velocity $v\in UM$ as $\gamma_v(t)$.
Here $UM$ is the unit tangent bundle of $M$.

\begin{definition}
For $v\in UM$, we define the \textit{Busemann function} $b_v : M\rightarrow \mathbb{R}$ on $M$ by
\begin{equation}
b_v(x)=\lim_{t\rightarrow \infty}\left\{d(\gamma_v(t), x)-t\right\}.
\end{equation}
We call the level hypersurfaces of the Busemann function the \textit{horosphere}.
We set $H_v:=b_v^{-1}(0)$.
\end{definition}
 
The Busemann function $b_v$ is uniformly continuous, and particularly Lipschitz continuous.
From Lipschitz continuity, $b_v$ is almost everywhere differentiable.

\begin{definition}
Let $\gamma_1$ and $\gamma_2$ be two geodesics in $M$ parametrized by its arc-length $t$.
We call that $\gamma_2$ is \textit{asymptotically equivalent}  to $\gamma_1$, if $d(\gamma_1(t), \gamma_2(t))$ is bounded above in $t>0$.
If $\gamma_1=\gamma_v$ and $\gamma_2=\gamma_w$ for $v, w\in UM$, we say also that $w$ is asymptotically equivalent to $v$.
\end{definition}

\begin{proposition}
We assume that $(M, g)$ has no conjugate points.
Then, the following holds:
\begin{enumerate}
\item The asymptotic relation on the set of all geodesics in $M$ (or on the unit tangent bundle $UM$) is an equivalence relation (\cite[Cor.1.10]{Inn}).
\item Let $v, w \in UM$. If $w$ is asymptotically equivalent to $v$, then $b_v-b_w$ is constant on $M$ (\cite[Prop.1.11.]{Inn}). 
\item For arbitrary $v\in UM$ and $p\in M$, there exists $w\in T_pM$ such that $w$ is asymptotically equivalent to $v$ (\cite[Prop.1.12.]{Inn}).
\item The Busemann function $b_v$ is of $C^1$-class and its gradient vector field $\nabla b_v$ is given by $\nabla b_v(q)=-\gamma_w'(0)$, where $w\in T_qM$ is asymptotically equivalent to $v$ (\cite[p.241, Proposition 1]{Esch}).
In particular, $\nabla b_v$ is a unit normal vector field on a horosphere.
\end{enumerate}
\end{proposition}

We consider the shape operator of a horosphere $H_v$ with respect to the unit normal vector field $\xi=-\nabla b_v$.
Then the second fundamental form is given by the Hessian of the Busemann function and the mean curvature is the minus sign of the Laplacian $\sigma_H=-\Delta b_v$.

\begin{remark}
We consider the mean curvature of the horospheres,
so we require the Busemann function to be of $C^2$-class.
For instance, if $(M, g)$ has no focal points,
this requirement is guaranteed (\cite[p.246, Theorem 2 (i)]{Esch}).
The definition of a manifold having no focal points can be found in \cite[p.312]{Inn}.
If a manifold has non-positive sectional curvature, then it has no focal points.
Moreover, if a manifold has no focal points, then it has no conjugate points \cite[p.311]{Inn}.
The Busemann function on non-compact harmonic manifolds,
which we will define in the following section,
is also $C^2$-class, and even real analytic \cite{RS}.
\end{remark}

Finally, we mention that horospheres are obtained as limits of geodesic spheres.
For an arbitrary $v\in UM$, consider the geodesic sphere $S(\gamma_v(t), t)$.
We choose a point $p_t\in S(\gamma_v(t), t)$ for each $t >0$, which is characterized as $d(\gamma_v(t), p_t)-t=0$.
We note that $\gamma_v(0)$ is always contained within $S(\gamma_v(t), t)$ for any $t>0$.
Roughly speaking, we can understand that a sequence of such a point $p_t$ converges to a point $p$ satisfying $\lim_{t\rightarrow \infty}(d(\gamma_v(t),p)-t)=0$,
which implies $b_v(p)=0$, i.e., $p\in H_v$.
See also \cite[p.8]{EsOs}.


\subsection{Harmonic manifolds}

\begin{definition}
Let $(M^n, g)$ be a complete Riemannian manifold.
When the volume density function $\omega_p:=\sqrt{\det(g_{ij})}$ in the normal coordinate neighborhood $\{x_1, \ldots, x_n\}$ around each point $p\in M$ depends only on the distance $r=d(p,\,\cdot\,)$, we call $(M, g)$ a \textit{harmonic manifold}.
\end{definition}

There are several assertions equivalent to the definition of harmonic manifolds.
Here, we list only those relevant to our discussion in the following theorem.
For further details, please refer to the references.

\begin{theorem}[cf. \cite{Sz}, {\cite[6.21 Proposition]{Besse}},  {\cite[Chap.II]{RWW}}]
\label{equiv_harmonicmfd}
Let $(M, g)$ be a complete Riemannian manifold.
The following are equivalent to each other.
\begin{enumerate}
\item $(M, g)$ is harmonic.
\item The volume density $\Theta_p(\exp_p ru)$ of $S(p;r)$ is a radial function, i.e., $\Theta_p(\exp_p ru)$ does not depend on $u\in U_pM$.
\item The mean curvature $\sigma_p(\exp_p ru)$ of $S(p;r)$ is a radial function, i.e., $\sigma_p(\exp_p ru)$ does not depend on $u\in U_pM$.
\item the averaging operator $\mathcal{M}_p$ commutes with the Laplace-Beltrami operator $\Delta$;\, $\Delta\circ\mathcal{M}_p = {\mathcal{M}}_p\circ\Delta$.
Here, for a smooth function $f$, $\mathcal{M}_p(f)$ is a smooth radial function on $M$ whose value is the average of $f$ on $S(p;r)$;
\begin{equation}\label{avrgop}
\mathcal{M}_p(f)(x) := \frac{1}{\int_{S(p;r)}dv_{S(p;r)}}\, \int_{x\in S(p;r)} f(x)\,d v_{S(p;r)}.
\end{equation}
\end{enumerate}
\end{theorem}

\begin{remark}
Due to the invariance of the volume density under the canonical geodesic involution,
$\Theta_p(q)$ is a function $\Theta(r(q))$ of the distance $r$ independent of $p\in M$ (see \cite[6.12 Lemma]{Besse}).
Similarly, from Lemma \ref{dens_meancurv}, $\sigma_p(q)$ is also a radial function $\sigma(r(q))$ independent of $p$.
\end{remark}

From \eqref{ledgerformulaelse}, each harmonic manifold is Einstein,
i.e., $\mathrm{Ric}=\dfrac{s}{n} g$, where $s$ is the scalar curvature of $(M, g)$.

\begin{lemma}\label{meancurv_Q}
Let $(M, g)$ be a non-compact harmonic manifold and $Q$ be the volume entropy of $(M, g)$.
Then, all horospheres in $M$ have constant mean curvature, given by
\begin{equation}\label{HessB_Q}
\sigma_{H_v}=-\Delta b_v=Q.
\end{equation}
\end{lemma}

\begin{proof}
If $M$ is a harmonic manifold, then from \eqref{def_volent} and \eqref{volBall},
the volume entropy is given by
\begin{equation*}
Q=\lim_{r\rightarrow \infty}
\dfrac{\log\left(\omega_{n-1}\int^r_0 \Theta(t)\,dt\right)}{r}=\lim_{r\rightarrow \infty}
\dfrac{\Theta'(r)}{\Theta(r)}.
\end{equation*}
Here, the second equality in the above equation is obtained by applying l'H\^{o}pital's rule twice.
As mentioned in section \ref{S_busemann},
a horosphere $H_v$ is obtained as the limit of the geodesic sphere $S(\gamma_v(t), t)$.
From Lemma \ref{dens_meancurv}, we find that the volume entropy $Q$ is the limit of the mean curvature of the geodesic sphere with the inward unit normal vector field.
Hence, we obtain \eqref{HessB_Q}.
\end{proof}


In section 1, we briefly mentioned the spherical Fourier transform, defined using spherical functions, which are eigenfunctions of the radial part $\Delta^{\mathrm{rad}}$ of the Laplace-Beltrami operator.
For any chosen $v \in UM$, if we set
\begin{equation*}
\varphi(x) = \exp\left\{-\left(\dfrac{Q}{2}-i\lambda\right) b_v(x)\right\},
\end{equation*}
then $\Delta \varphi(x) = \left(\dfrac{Q^2}{4}+\lambda^2\right) \varphi(x)$ holds.
From theorem \ref{equiv_harmonicmfd} (iv), we find that the function $\mathcal{M}_p\varphi$ is an eigenfunction of $\Delta^{\mathrm{rad}}$.
Therefore, we can define the spherical Fourier transform on a harmonic manifold.

\section{Hypergeomeric type}

In this section, we assume that $(M^n, g)$, $n \geq 3$ is a non-compact harmonic manifold with volume entropy $Q>0$.

\begin{definition}\label{HMHGT}
We call $M$ a harmonic manifold of \textit{hypergeometric type} when the eigenfunction equation \eqref{LEeq} of the radial part of the Laplace-Beltrami operator $\Delta^{\mathrm{rad}}$ on $M$ is transformed into the hypergeometric equation \eqref{HGeq} by a certain variable transformation $r=r(z)$.
\end{definition}

\begin{proposition}\label{HG_transf}
Suppose that there exists a variable transformation $z = z(r)$ which takes the equation \eqref{LEeq} with an arbitrary real parameter $\lambda$ associated to $\Phi(r)$ into the equation \eqref{HGeq} associated to a certain function $f(z)$. 
Then, it is concluded that for some constant $\ell>0$
\begin{enumerate}
\item $a = \bar{b} = \dfrac{1}{\ell} \left(\dfrac{Q}{2} \pm i \lambda\right)$ and $c= \dfrac{n}{2}$,
\item the transformation $z= z(r)$ must be $z = - \sinh^2 \dfrac{\ell r}{2}$.
\end{enumerate}
\end{proposition}

A similar claim was also proven in \cite[11. Appendix]{Itoh-S2020},
but we restate the statement and its proof with slight modifications here.

\begin{proof}
Suppose that $\Phi(r)$ is converted into $f(z)$ by $z=z(r)$ as $f(z(r))= \Phi(r)$.
Then, we have
\begin{equation*}
\frac{d\Phi}{d r} =\frac{d f}{d z}\, \frac{d z}{d r},\quad
\frac{d^2\Phi}{d r^2} =\frac{d^2 f}{d z^2}\left(\frac{d z}{d r}\right)^2 + \frac{d f}{d z}\, \frac{d^2 z}{d r^2},
\end{equation*}
so that
\begin{multline*}
\left(\frac{d^2}{dr^2}+ \sigma(r)\frac{d}{dr}\right) \Phi + \left(\frac{Q^2}{4}+\lambda^2\right) \Phi 
\\
= \frac{d^2 f}{d z^2}\, \left(\frac{d z}{d r}\right)^2 +  \left(\frac{d^2 z}{d r^2} + \sigma(r)\, \frac{d z}{d r} \right)\frac{d f}{d z}\, + \left(\frac{Q^2}{4}+\lambda^2\right)\, f = 0.
\end{multline*}
This equation must turn out to be the equation \eqref{HGeq}.
Hence, since $\frac{Q^2}{4}+\lambda^2 \not= 0$, it follows $ab \not=0$.
Therefore we get
\begin{align}\label{firsteqn}
\left(\frac{d z}{d r}\right)^2 
=& - \frac{\left(\frac{Q^2}{4}+\lambda^2\right)}{ab}\, z(1-z),\\
\label{secondeqn}
\frac{d^2 z}{d r^2} + \sigma(r)\, \frac{d z}{d r} 
=& - \frac{\left(\frac{Q^2}{4}+\lambda^2\right)}{ab}\,\{c- (a+b+1)z\}.
\end{align}
From the following considerations,
the solutions to the differential equation \eqref{firsteqn} can be considered in two ways depending on the sign of $ab$:
\begin{enumerate}
\item[(i)] when $ab<0$, $z=\sin^2\left(\frac{\ell r}{2}+C\right)$,
\item[(ii)] when $ab>0$, $z=-\sinh^2\left(\frac{\ell r}{2}+C\right)$,
\end{enumerate}
where $C$ is an arbitrary constant, and case (i) is not appropriate, and case (ii) with $C=0$ is appropriate for our argument.

\medskip

\noindent{\bf Case (i)}\, When $ab < 0$, one has from \eqref{firsteqn} $z(1-z) \geq 0$ and hence $0 \leq z \leq 1$.
Then, \eqref{firsteqn} is solved by setting $z = \sin^2 t$ for a function $t=t(r)$ and inserting it to have
\begin{equation*}
\frac{d z}{d r}= 2\sin  t \cos  t \cdot \frac{dt}{dr}
\end{equation*}
and
\begin{equation*}
\ell^2 z(1-z) = \ell^2 \sin^2 t \cos^2  t,\quad\mbox{where}\ 
\ell = \sqrt{-\dfrac{\frac{Q^2}{4}+ \lambda^2}{ab}}
\end{equation*}
and then $\left(\frac{dt}{dr}\right)^2 = \frac{\ell^2}{4}$, namely one may put $t = \frac{\ell r}{2}+C$ and $z= \sin^2\left( \frac{\ell r}{2}+C\right)$.
The left-hand side of \eqref{secondeqn} is then written
\begin{equation}\label{lefthand}
\frac{d^2 z}{d r^2} + \sigma(r)\, \frac{d z}{d r}
= \frac{\ell}{2}\left\{\ell \cos (\ell r+2C)+ \sigma(r)\sin (\ell r+2C)\right\},
\end{equation}
and the right-hand side is
\begin{equation}\label{righthand}
\ell^2\{c-(a+b+1)z\} = \ell^2 \left\{ c - \frac{1}{2}(a+b+1) + \frac{1}{2}(a+b+1) \cos (\ell r+2C)\right\}.
\end{equation}
When $r \rightarrow 0$, from \eqref{meancurv_Q},
so if $C\ne 0$ then \eqref{righthand} converges to a finite value,
but \eqref{lefthand} diverges.
Hence we find that the constant $C$ can be set to $0$.
If $C=0$, then \eqref{lefthand} tends to $\frac{\ell}{2}\{\ell + \ell(n-1)\} = \frac{\ell^2 n}{2}$ and \eqref{righthand} goes to $\ell^2 c$ as $r \rightarrow 0$ so that $c = \frac{n}{2}$.
Since \eqref{lefthand} equals \eqref{righthand}, one has
\begin{align}
\frac{\sigma(r)}{2\ell}\, \sin \ell r
=& c - \frac{1}{2} (a+b+1) + \frac{1}{2}(a+b)\cos \ell r \nonumber \\
=& \label{equation-x} \frac{n-1}{2} + \frac{1}{2} (a+b)(\cos \ell r - 1).
\end{align} 
If $a+b >0$, then we can choose $r$ such that $\frac{3\pi}{2} < \ell r < 2\pi$ and satisfying 
$\cos \ell r > 1-\frac{n-1}{a+b}$.
Then, the left hand of \eqref{equation-x} is negative, while the left hand is positive. This is a contradiction.

On the other hand, if $a+b \leq 0$, then choose $r > 0$ such that $\cos \ell r < \frac{1}{2}$ and $\sin \ell r < 0$ (i.e., $\frac{5\pi}{3}<\ell r < 2\pi$ ).
Then, the left-hand side of \eqref{equation-x} is negative, while the right-hand side is positive.
This is a contradiction.

\medskip

\noindent{\bf Case (ii)}\, When $a b > 0$, $z(1-z) \leq 0$, so either $z \leq 0$ or $z \geq 1$.
We get $z= -\sinh^2 \left(\frac{\ell r}{2}+C\right)$ for the case of $z \leq 0$ or $z = \cosh^2 \left(\frac{\ell r}{2}+C\right)$ for $z \geq 1$ in a similar manner as case (i). Here
\begin{equation}\label{ell}
\ell = \sqrt{\frac{\frac{Q^2}{4}+ \lambda^2}{ab}}>0.
\end{equation}
From the following lemma, even if we consider only the former, generality is not lost.

\begin{lemma}\label{z->1-z}
If $u(z)$ is a solution to \eqref{HGeq},
then $v(z):=u(1-z)$ is a solution to the hypergeometric differential equation of another type:
\begin{equation*}
z(1-z)\,f''(z)+\{c_1-(a_1+b_1+1)z\}\,f'(z)-a_1b_1\,f(z)=0,
\end{equation*}
$a_1=a$, $b_1=b$, $c_1=a+b+1-c$.
\end{lemma}

Moreover, we find that $C$ must be $0$ in a similar manner as case (i).

The left hand of \eqref{secondeqn} is written
\begin{equation}\label{lefthand-2}
\frac{d^2 z}{d r^2} + \sigma(r)\, \frac{d z}{d r} = - \frac{\ell}{2}\left(\ell \cosh \ell r+ \sigma(r)\sinh \ell r\right)
\end{equation}
and the right hand is
\begin{equation}\label{righthand-2}
-\ell^2\{c-(a+b+1)z\} = -\ell^2 \left\{ c + (a+b+1) \sinh^2 \frac{\ell r}{2}\right\}.
\end{equation}
When $r\rightarrow 0$, \eqref{righthand-2} tends to $-\ell^2 c$.
On the other hand \eqref{lefthand-2} tends to $- \frac{\ell^2 n}{2}$.
Thus $c = \frac{n}{2}$.

From the equality of the right hand sides of \eqref{lefthand-2} and \eqref{righthand-2}, we obtain
\begin{equation}\label{prop_meancurv}
\sigma(r)=\dfrac{\ell}{\sinh \ell r}
\left\{
2c-(a+b+1)+(a+b)\cosh \ell r
\right\},
\end{equation}
so we have $(a + b) \ell = Q$,
since $\lim_{r\rightarrow\infty}\sigma(r) = \sigma_{H_v}=Q$ from lemma \ref{meancurv_Q}.
Therefore, from \eqref{ell}, we have $a = \frac{1}{\ell}\left(\frac{Q}{2} + i \lambda\right)$ and $b = \frac{1}{\ell}\left(\frac{Q}{2} - i \lambda\right)$.
\end{proof}

\begin{theorem}\label{volumedensmeancurv}
Let $(M^n, g)$ be of hypergeometric type.
Then the mean curvature $\sigma(r)$ and the volume density $\Theta(r)$ of a geodesic sphere are described respectively by
\begin{align}
\label{HG_meancurv}
\sigma(r)=& \frac{\ell(n-1)}{2}\coth\frac{\ell r}{2}
+\left\{Q-\frac{\ell(n-1)}{2}\right\}\tanh\frac{\ell r}{2},\\
\label{HG_density}
\Theta(r)=& \left(\frac{2}{\ell}\right)^{n-1} \left(\sinh\frac{\ell r}{2}\right)^{n-1}
\left(\cosh\frac{\ell r}{2}\right)^{\frac{2Q}{\ell}-(n-1)}.
\end{align}
\end{theorem}

\begin{proof}
From the proof of proposition \ref{HG_transf},
it follows that the mean curvature of the geodesic spheres in $M$ is necessarily expressed in the form of \eqref{prop_meancurv}.
Since the parameters $a, b$ and $c$ are determined by proposition \ref{HG_transf} (1), an easy calculation shows that $\sigma(r)$ is given in the form \eqref{HG_meancurv}.

By Lemma \ref{dens_meancurv} and \eqref{HG_meancurv},
it is immediately obtained from straightforward integration that $\Theta(r)$ is given in the form \eqref{HG_density} with a certain constant $k$.
We can compute the coefficient $k$ as follows.
We write $\Theta(r)$ as
\begin{align*}
\Theta(r)=&k\left(\tanh\frac{\ell r}{2}\right)^{n-1}
\left(\cosh\frac{\ell r}{2}\right)^{\frac{2Q}{\ell}}
\end{align*}
and when we expand the right-hand side into series of $\tanh \frac{\ell r}{2}$ and $\left(\cosh \frac{\ell r}{2}\right)^{\frac{2Q}{\ell}}$, we have
\begin{align}
\Theta(r)=&k\left\{
\frac{\ell r}{2}-\frac{1}{3}\left(\frac{\ell r}{2}\right)^3+\cdots
\right\}^{n-1}
\left(1+\frac{Q}{2}\cdot\frac{\ell}{2}r^2+\cdots\right)\notag\\
=&k\left\{
\left(\frac{\ell}{2}\right)^{n-1}r^{n-1}-\frac{n-1}{3}\left(\frac{\ell}{2}\right)^{n+1}r^{n+1}+\cdots\right\}
\left(1+\frac{Q}{2}\cdot\frac{\ell}{2}r^2+\cdots\right)\notag\\
=&k\left\{
\left(\frac{\ell}{2}\right)^{n-1}r^{n-1}+\left(\frac{Q}{\ell}-\frac{n-1}{3}\right)\left(\frac{\ell}{2}\right)^{n+1}r^{n+1}+\cdots\right\}\label{volden-seqex}.
\end{align}
From \eqref{ledgerformula0}, we obtain $k=\left(\frac{2}{\ell}\right)^{n-1}$.
\end{proof}

Conversely, from the reverse argument of the proof of proposition \ref{HG_transf},
if $\sigma(r)$ is given in the form \eqref{HG_meancurv},
then it follows that $M$ is of hypergeometric type.
More generally, the following holds.

\begin{theorem}[cf. \cite{Itoh-S2020}]\label{voldensfctheoremHGT}
The following are equivalent each other.
\begin{enumerate}
\item $(M, g)$ is of hypergeometric type,
\item $\displaystyle \sigma(r)= c_1\coth\frac{\ell r}{2}
+c_2\tanh\frac{\ell r}{2}$,
\item $\displaystyle \Theta(r)= k \left(\sinh\frac{\ell r}{2}\right)^{c_1}
\left(\cosh\frac{\ell r}{2}\right)^{c_2}$.
\end{enumerate}
Here $\ell, k, c_1$ are some positive constants and $c_2$ is a constant satisfying $c_1+c_2>0$.
\end{theorem}

\section{Proof of Theorem \ref{maintheorem}}

In this section, we provide the proof of our main theorem.
Let $(M, g)$ be a harmonic manifold of hypergeometric type with volume entropy $Q$ and the metric $g$ is normalized as $\mathrm{Ric}=-(n-1)$.
From theorem \ref{volumedensmeancurv},
the density function $\Theta(r)$ is given by \eqref{HG_density}.

The proof that the volume entropy of $M$ is bounded from above follows the same argument as \cite[Theorem 1.7]{Itoh-S2020}.
According to the Bishop volume comparison theorem \cite[IV, Theorem 3.1 (2) b)]{Sakai},
the density function satisfies the following inequality $\Theta(r)\le \left(\sinh r\right)^{n-1}$,
where the right-hand side is the density function of a hyperbolic space $\mathbb{R}H^n(-1)$ with sectional curvature $-1$.
By examining its behavior as $r\rightarrow \infty$,
we obtain $Q \le n - 1$.
Equality holds if and only if $M$ is isometric to $\mathbb{R}H^n(-1)$ (see \cite[IV, Corollary 3.2. (2)]{Sakai}).

Next, we show that the volume entropy of $M$ is bounded from below.
From the Ledger's formula \eqref{ledgerformulaelse} and \eqref{volden-seqex},
we have
\begin{equation*}
-\dfrac{1}{3}\mathrm{Ric}_g=\left(\dfrac{2}{\ell}\right)^{n-1}\cdot \left(
\dfrac{Q}{\ell}-\dfrac{n-1}{3}
\right)\left(\dfrac{\ell}{2}\right)^{n+1},
\end{equation*}
Hence we have
\begin{equation*}
Q=\left(\ell+\dfrac{2}{\ell}\right)\dfrac{n-1}{3},
\end{equation*}
from the AM--GM inequality, so that $Q\ge 2\sqrt{\ell\cdot\frac{2}{\ell}}\cdot\frac{n-1}{3}=\frac{2\sqrt{2}}{3}(n-1)$.

\section{The volume entropy of Damek-Ricci spaces}\label{DR_4ex}

The upper bound of the volume entropy is obtained from the volume comparison theorem,
and the space that attains this value is determined.
However, the lower bound is obtained straightforwardly from the simple AM-GM inequality.
Is this lower estimation sharp?
In this section, we show that among Damek-Ricci spaces,
there exist only four examples where the volume entropy attains its lower bound.

First, let us provide a brief introduction of Damek-Ricci spaces.
See \cite{BTV} for details.
A Damek-Ricci space $S$ is a one-dimensional extension of a 2-step nilpotent group $N$,
called the generalized Heisenberg group, with a certain left-invariant metric $g$.
Let $\mathfrak{z}$ be the center of the Lie algebra $\mathfrak{n}$ of $N$ and $\mathfrak{v}$ be its orthogonal complement.
By the definition of the generalized Heisenberg group, $\mathfrak{z}$ has a Euclidean structure with a quadratic form $q$, which is defined as $q(Z):=-\langle Z, Z\rangle$ using the inner product, and $\mathfrak{v}$ is a Clifford module over the Clifford algebra $Cl(\mathfrak{z},q)$.
From the classification of representations of Clifford algebras,
we obtain the classification $\mathfrak{n} = \mathfrak{v} \oplus \mathfrak{z}$.
In particular, $k:= \dim \mathfrak{v}$ and $m:= \dim \mathfrak{z}$ can be classified as shown in the table of \cite[p.23]{BTV}.

A Damek-Ricci space $(S, g)$ is a harmonic manifold whose Ricci tensor is
\begin{equation}\label{DR_Ric}
\mathrm{Ric}=-\left(m+\frac{k}{4}\right)g.
\end{equation}
Since the volume density $\Theta(r)$ is given by
\begin{equation*}
\Theta(r)=2^{k+m} \left(\sinh\frac{r}{2}\right)^{k+m}
\left(\cosh\frac{r}{2}\right)^m,
\end{equation*}
$S$ is of hypergeometric type, and its volume entropy is $Q=\frac{k}{2}+m$ \cite[(1.16)]{ADY}.
Since, from \eqref{DR_Ric}, the Ricci tensor equals
\begin{equation*}
\mathrm{Ric}
=-(k+m)\cdot\frac{m+\frac{k}{4}}{k+m}g
=-(n-1)\cdot\frac{m+\frac{k}{4}}{k+m}g,
\end{equation*}
so rescaling $g\mapsto c^2g$ for $c=\sqrt{\dfrac{\frac{k}{4}+m}{k+m}}$,
we have $\mathrm{Ric} = -(n-1)$.
Furthermore, from lemma \ref{Q_rescaling},
the volume entropy for this metric is given by
\begin{equation*}
Q=\left(m+\frac{k}{2}\right)\sqrt{\dfrac{k+m}{\frac{k}{4}+m}}.
\end{equation*}
Consider the case where this value satisfies
\begin{equation*}\label{Qinf}
\left(m+\frac{k}{2}\right)\sqrt{\dfrac{k+m}{\frac{k}{4}+m}}
=\dfrac{2\sqrt{2}(n-1)}{3}
=\dfrac{2\sqrt{2}(k+m)}{3}.
\end{equation*}
Squaring both sides of the above equation and rearranging,
we obtain $\left(m-\frac{k}{2}\right)^2=0$,
so we have $k=2m$.
From the classification table of the generalized Heisenberg algebra in \cite[p.23]{BTV},
we find that there are only four cases $m=1, 2, 4, 8$.

\begin{remark}
In the case of $m = 1$,
$S$ is the complex hyperbolic space $\mathbb{C}H^2$ \cite[p.79]{BTV}.
In the case of $m = 2$,
$S$ has seven dimensions, which is the lowest dimension of non-symmetric Damek-Ricci spaces \cite[p.110]{BTV}.
\end{remark}

The forthcoming challenge is to establish a general characterization of spaces that yield the lower bound for the volume entropy.

\section*{Acknowledgements}

The author would like to express gratitude to Prof. Hemangi M. Shah for providing the opportunity to stay at the Harish-Chandra Research Institute in India from February 23rd to March 4th, 2024.
During the visit, fruitful discussions were held on the topics covered in this paper as well as related subjects.
Additionally, the author appreciates the opportunity to deliver lectures on this topic.
The author is also grateful to the referee for carefully reading the manuscript and providing valuable comments, which helped improve the clarity and quality of this paper.

\end{document}